\documentclass[11pt]{amsart}
\usepackage[usenames,dvipsnames,svgnames,table]{xcolor}
\usepackage{amssymb,amsfonts,amsmath,amscd}
\usepackage{verbatim}
\usepackage[all]{xy}
\usepackage{hyperref}

\newtheorem{theorem}{Theorem}[section]
\newtheorem{lemma}[theorem]{Lemma}     
\newtheorem{corollary}[theorem]{Corollary}
\newtheorem{proposition}[theorem]{Proposition}

\newtheorem*{theorem*}{Theorem}
\newtheorem*{corollary*}{Corollary}

\theoremstyle{definition}
\newtheorem*{remark*}{Closing Remark}

\newtheorem{remark}[theorem]{Remark}

\newcommand{\espai}{\phantom{++}}

\newcommand{\mbc}{\mbox{$\mathbb{C}$}}

\newcommand{\mfm}{\mbox{$\mathfrak{m}$}}
\newcommand{\mfn}{\mbox{$\mathfrak{n}$}}
\newcommand{\mfp}{\mbox{$\mathfrak{p}$}}
\newcommand{\mfq}{\mbox{$\mathfrak{q}$}}

\newcommand{\rad}{\mbox{${\rm rad}$}}

\newcommand{\height}{\mbox{${\rm ht}$}}
\newcommand{\grade}{\mbox{${\rm grade}$}}
\newcommand{\depth}{\mbox{${\rm depth}$}}
\newcommand{\length}{\mbox{${\rm length}$}}
\newcommand{\rank}{\mbox{${\rm rank}$}}
\newcommand{\pd}{\mbox{${\rm proj}\,{\rm dim}\,$}}

\newcommand{\Ext}{\mbox{${\rm Ext}$}}

\newcommand{\Ann}{\mbox{${\rm Ann}$}}
\newcommand{\Spec}{\mbox{${\rm Spec}$}}

\newcommand{\Gor}{\mbox{${\rm Gor}$}}
\newcommand{\Syz}{\mbox{${\rm Syz}$}}

\title[Regular local rings of dimension 4] {Regular local rings of
dimension four and Gorenstein syzygetic prime ideals}

\author{Francesc Planas-Vilanova}

\date{\today}


\keywords{Regular local rings, Gorenstein rings, syzygetic ideals, 
homology of Andr\'e-Quillen.\newline
This work is partially supported by the
2017SGR932 and PID2019-103849GB-I00/AEI / 10.13039/501100011033}

\begin{document}

\maketitle

{\centering\footnotesize Dedicated to the memory of Wolmer
  V. Vasconcelos.\par}

\begin{abstract}
Let $R$ be a Noetherian local ring. We prove that $R$ is regular of
dimension at most 4 if, and only if, every prime ideal, defining a
Gorenstein quotient ring, is syzygetic. We deduce a characterization
of these rings in terms of the Andr\'e-Quillen homology.
\end{abstract}

\section{Introduction}

Let $R$ be a Noetherian commutative ring and $I$ an ideal of $R$.  Let
$\alpha:{\bf S}(I)\to {\bf R}(I)$ be the graded surjective morphism
from the symmetric algebra of $I$ to the Rees algebra of $I$. The
ideal $I$ is said to be syzygetic if the second component
$\alpha_2:{\bf S}_2(I)\to I^2$ is an isomorphism; it is said to be of
linear type if $\alpha$ is an isomorphism. Ideals generated by a
regular sequence are of linear type, hence syzygetic.  Noetherian
rings of global dimension at most 1, 2 and 3, were recently
characterized in \cite{planasJPAA} in terms of the syzygetic and
linear type conditions. Recall that the global dimension of $R$ is
defined as the supremum of the projective dimensions of all
$R$-modules. For a Noetherian ring $R$, having global dimension at
most $N$ is equivalent to $R_{\mathfrak{m}}$ being regular local with
$\dim(R_{\mathfrak{m}})\leq N$, for every maximal ideal $\mfm$ of $R$.
The purpose of this note is to extend \cite{planasJPAA} to dimension
4, but restricted to the local case.

\begin{theorem*}\label{main}
Let $(R,\mfm)$ be a Noetherian local ring. The following conditions
are equivalent:
\begin{itemize}
\item[$(i)$] $R$ is regular and $\dim(R)\leq 4$;
\item[$(ii)$] Every prime ideal $\mfp$ of $R$, with $R/\mfp$
  Gorenstein, is syzygetic;
\end{itemize}
\end{theorem*}

The main ingredient in the implication $(i)\Rightarrow (ii)$ is the
following result of Herzog, Simis and Vasconcelos: if $R$ is regular
local, with $1/2\in R$, then a Gorenstein ideal of height 3 is
syzygetic (see \cite[Proposition~2.8]{hsv}). In order to avoid the
condition $1/2\in R$, we give a different approach borrowing ideas
from Ulrich in \cite[(2.2)]{ulrich}. This is done in
Lemma~\ref{gorenstein}

As for the reverse implication, suppose that $(ii)$ holds. Using
another nice result of Herzog, Simis and Vasconcelos one deduces that
$R$ is regular. Indeed, a Noetherian local ring whose maximal ideal is
syzygetic is regular (see the proof of Corollary~3.8 and,
particularly, Proposition~2.5, in \cite{hsv}).

Thus, one must show that any regular local ring $R$ of dimension at
least 5 admits a prime ideal $\mfp$, whose factor ring $R/\mfp$ is
Gorenstein, and such that $\mfp$ is not syzygetic. We first address
the five dimensional case. Our candidate, call it $I$, is inspired by
the following affine example: consider the most simple five
dimensional Gorenstein curve which is not a complete intersection
(\cite[Theorem~4.4]{hunekeCM}). Then take a minimal system of
generators of its equations, namely, the kernel of the ring
homomorphism $\mbc[[X,Y,Z,T,U]]\to\mbc[[V]]$, which sends $X,Y,Z,T,U$
to $V^6,V^7,V^8,V^9,V^{10}$. On substituting the variables by the
regular parameters $x,y,z,t,u$ of the regular local ring $R$, we
obtain a minimal system of generators of the candidate ideal $I$.

Lemma~\ref{perfect} proves that $I$ is a perfect ideal of height 4,
such that $R/I$ is Gorenstein, and such that $I$ is not syzygetic. To
see that $I$ is perfect we use the well-known acyclicity criterion of
Buchsbaum and Eisenbud. We would like to stress here that our proof
holds in any regular local ring, non necessarily an algebra
essentially of finite type over a field. In that sense, Singular
(\cite{dgps}) is of great help in finding and checking products of
matrices, as well as guessing which minors will conform the regular
sequences of the required length. However, our specific full proof of
the perfectness of $I$ can not be deduced, at least just by transport,
from Singular.

The key point of the paper is to show that $I$ is in fact a prime
ideal. First, we reduce to the complete case. Then we take any
associated prime $\mfp$ to $I$, necessarily different from the maximal
ideal. Since $R$ is complete, the integral closure of $R/\mfp$ is a
DVR. Using the valuation corresponding to $R/\mfp$, we are able to
deduce the equality $I=\mfp$. This is done in Proposition~\ref{dim5}.

Once the result is proved in dimension 5, we extend the example in any
arbitrary higher dimension in Corollary~\ref{dimn}.

The paper finishes with a characterization of Noetherian local rings
which are regular and of dimension at most 4 in terms of the
Andr\'e-Quillen homology (see Corollary~\ref{corollaryAQ}).

For any unexplained notation we refer to \cite{matsumura}, \cite{bh},
\cite{hsv}, \cite{andre} and \cite{quillen}.

\section{Proof of the main result}

\begin{lemma}\label{gorenstein}
Let $(R,\mfm)$ be a Gorenstein local ring. Let $I$ be an ideal of
height 3, generically a complete intersection, having finite
projective dimension, and such that $R/I$ is Gorenstein. Then $I$ is
syzygetic.
\end{lemma}
\begin{proof}
Let $S=R/I$ and let $H_1(I)$ denote the first Koszul homology group
associated to a minimal system of generators of $I$. By
\cite[Proposition~1.2.13]{bh}, for every associated prime $\mfq$ to
$H_1(I)$,
\begin{eqnarray*}
  \depth(H_1(I))\leq
  \dim(S/\mfq)\leq\dim(S/\Ann_S(H_1(I)))=\dim(H_1(I))\leq \dim S.
\end{eqnarray*}
By hypothesis, $I$ is in the linkage class of a complete intersection,
so $H_1(I)$ is a Cohen-Macaulay module of maximum dimension (see,
e.g., \cite[Remark~1.3, Theorem~1.14 and Example~2.2]{hunekeAJM}).
Therefore $\dim(S/\mfq)=\dim(S)$ and every associated prime $\mfq$ of
$H_1(I)$ is a minimal prime of $S$. Hence the set of zero divisors of
$H_1(I)$ is included in the set of zero divisors of $S$, which implies
$H_1(I)$ is torsion-free. On the other hand, by \cite[15.12]{andre} or
\cite[Discussion before Proposition~2.5]{hsv}, one has the exact
sequence:
\begin{eqnarray*}
0\to H_2(R,S,S)\to H_1(I)\to F\otimes S\to I/I^2\to 0,
\end{eqnarray*}
where $F\to I\to 0$ is a minimal free presentation of $I$ and
$H_2(R,S,M)$ stands for the second Andr\'e-Quillen homology group of
the $R$-algebra $S$ with coefficients in the $S$-module $M$ (see
\cite{andre}, \cite{quillen}, \cite{mr}, \cite{iyengar}).  Since $I$
is generically a complete intersection at its associated primes, for
every associated prime $\mfp$ to $I$,
\begin{eqnarray*}
  H_2(R,S,S)_{\mathfrak{p}}=H_2(R_{\mathfrak{p}},R_{\mathfrak{p}}/I_{\mathfrak{p}},
  R_{\mathfrak{p}}/I_{\mathfrak{p}})=0.
\end{eqnarray*}
Hence $H_2(R,S,S)$ is a $S$-torsion module, thus included in the
torsion of $H_1(I)$, which is zero. Therefore $H_2(R,S,S)=0$. Since
$H_2(R,R/I,R/I)\cong\ker(\alpha_2:{\bf S}_2(I)\to I^2)$, it follows
that $I$ is syzygetic (see, e.g., \cite[Corollaire~15.10]{andre} or
\cite[Corollaire~3.2]{planasMan}).
\end{proof}

\begin{proof}[Proof of Theorem $(i)\Rightarrow (ii)$.]
Suppose that $(R,\mfm)$ is regular local with $\dim(R)\leq 4$.  Then
$\mfm$ is generated by a regular sequence, so $\mfm$ is syzygetic
(see, e.g., \cite[Corollary~3.8]{hsv}); since $R$ is a UFD, then every
height 1 prime ideal is principal (generated by a nonzero divisor),
and so again syzygetic; furthermore, every height 2 Gorenstein prime
ideal in a regular local ring is generated by a regular sequence, thus
syzygetic (see, e.g. \cite[\S4]{hunekeCM}). Let $\mfp$ be a prime
ideal of height 3, such that $R/\mfp$ is Gorenstein. Since $R$ is
regular, $\mfp$ has finite projective dimension and $\mfp$ is
generically a complete intersection. By Lemma~\ref{gorenstein}, $\mfp$
is syzygetic.
\end{proof}

\begin{lemma}\label{perfect}
Let $(R,\mfm,k)$ be a regular local ring of dimension $5$. Let
$x,y,z,t,u$ be a regular system of parameters. Let $I$ be the ideal of
$R$ generated by
\begin{eqnarray*}
  &&f_1=y^2-xz\mbox{ , }f_2=yz-xt\mbox{ , }f_3=z^2-yt\mbox{,}\\
  && f_4=yt-xu\mbox{ , }f_5=zt-yu\mbox{ , }f_6=t^2-zu\mbox{,}\\
  && f_7=zu-x^3\mbox{ , }f_8=tu-x^2y\mbox{ , }f_9=u^2-x^2z.
\end{eqnarray*}
Then $I$ is a perfect ideal of height 4, such that $R/I$ is
Gorenstein, and such that $I$ is not syzygetic.
\end{lemma}
\begin{proof}
Let $\varphi_1=(f_1,\ldots,f_9)$ be the $1\times 9$ matrix given by
the nine aforementioned binomials:
\begin{eqnarray*}
\varphi_1=\left(\begin{array}{ccccccccc}
y^2-xz,yz-xt,z^2-yt,yt-xu,zt-yu,t^2-zu,zu-x^3,tu-x^2y,u^2-x^2z
\end{array}\right).
\end{eqnarray*}
Let $\varphi_2$, $\varphi_3$ and $\varphi_4$ be the $9\times 16$,
$16\times 9$ and $9\times 1$, matrices defined as:
\begin{eqnarray*}
\varphi_2=\left(\begin{array}{rrrrrrrrrrrrrrrr}
0& 0& -u&t& z&0&x^2&0& u& t& 0& 0& 0& 0& 0& 0\\
0& u& t& 0& -y&x^2&0& u& 0& -z&0&x^2&0& u& 0& 0\\
u& 0& 0& 0& x& 0& u& 0& 0& y&x^2&0& u& -t&0& 0\\
u& 0& 0& -y&0&0& 0& -t&-z&0&x^2&0& 0& -t&0&x^2\\
0& 0& -y&x& 0&  u& 0& 0& y& 0& 0& 0& -t&z&x^2&0\\
0& 0& x& 0& 0& 0& 0& y& 0& 0& 0& 0& z& 0& 0& -u\\
-z&-y&0& 0& 0& -t&-z&0& 0& 0& -u&-t&0& 0& 0& -u\\
0& x& 0& 0& 0&0& y& 0& 0& 0& 0& z& 0& 0& -u&t\\
x& 0& 0& 0& 0&y& 0& 0& 0& 0& z& 0& 0& 0& t& 0  
\end{array}\right),
\end{eqnarray*}
\begin{eqnarray*}
\varphi_3=\left(\begin{array}{rrrrrrrrr}
z& -t& 0& 0& y& 0& 0&  0& 0\\ 
-t&u&  0& 0& -z&0& 0&  0& 0\\  
u& 0&  y& -z&0& 0& 0&  0& xy\\ 
0& -x^2&-z&t& u& 0& 0&  0& -xz\\
x^2&0&  t& -u&0& 0& 0&  0& xt\\
0& 0&  0& 0& -x&z& -t& 0& 0\\  
-z&t&  0& 0& 0& -t&u&  0& 0\\  
0& -u& -x&0& 0& u& 0&  -z&-x^2\\
u& 0&  y& 0& 0& 0& -x^2&t& xy\\
0& 0&  0& -t&-u&x^2&0&  -u&0\\ 
-x&0&  0& 0& 0& -y&0&  0& -t\\ 
y& 0&  0& 0& x& 0& t&  0& u\\  
0& 0&  0& x& 0& 0& -u& y& 0\\  
0& 0&  x& -y&0& -u&0&  0& 0\\  
0& x&  0& 0& 0& 0& y&  0& z\\  
x& -y& 0& 0& 0& y& -z& 0& 0  
\end{array}\right),
\end{eqnarray*}
\begin{eqnarray*}
\varphi_4^\top=\left(
zt-yu,z^2-yt,u^2-xy^2,tu-x^2y,-t^2+zu,-yt+xu,-yz+xt,-zu+x^3,y^2-xz  
\right).
\end{eqnarray*}
Since $\varphi_1\cdot\varphi_2=0$, $\varphi_2\cdot\varphi_3=0$ and
$\varphi_3\cdot\varphi_4=0$, then
\begin{eqnarray*}
0\leftarrow R/I\leftarrow R=F_0\xleftarrow{\varphi_1}R^9=F_1
\xleftarrow{\varphi_2}R^{16}=F_2\xleftarrow{\varphi_3}R^9=F_3
\xleftarrow{\varphi_4}R=F_4\leftarrow 0
\end{eqnarray*}
is a complex of $R$-modules.  To see that this complex is exact, we
use the acyclicity criterion of Buchsbaum and Eisenbud (see, e.g.,
\cite[Theorem~1.4.12]{bh}). Set
$r_i=\sum_{j=i}^4(-1)^{j-i}\rank\,F_j$, so that $r_1=1$, $r_2=8$,
$r_3=8$ and $r_4=1$. Thus we have to prove that
$\grade(I_1(\varphi_1))\geq 1$, $\grade(I_8(\varphi_2))\geq 2$,
$\grade(I_8(\varphi_3))\geq 3$ and $\grade(I_1(\varphi_4))\geq 4$.

The ideal $(f_1,f_3+f_4,f_6+f_7,f_9,x)$ is equal to
$(y^2,z^2,t^2,u^2,x)$, so has grade $5$.  By
\cite[Corollary~1.6.19]{bh}, $f_1,f_3+f_4,f_6+f_7,f_9$ is an
$R$-regular sequence in $I=I_1(\varphi_1)$ of length four. In
particular, $\grade(I)\geq 4\geq 1$. Similarly, one has the equality:
\begin{eqnarray*}
  ((\varphi_4)_{9,1},(\varphi_4)_{2,1}-(\varphi_4)_{6,1},
  (\varphi_4)_{5,1}+(\varphi_4)_{8,1},(\varphi_4)_{3,1},x)=
  (y^2,z^2,t^2,u^2,x).
\end{eqnarray*}
Therefore $(\varphi_4)_{9,1},(\varphi_4)_{2,1}-(\varphi_4)_{6,1},
(\varphi_4)_{5,1}+(\varphi_4)_{8,1},(\varphi_4)_{3,1}$ is an
$R$-regular sequence in $I_1(\varphi_4)$ of length four and
$\grade(I_1(\varphi_4))\geq 4$.

In order to prove $\grade(I_8(\varphi_2))\geq 2$, we look for minors
of $\varphi_2$ with monic pure terms in one of the parameters. For
instance, up to sign, the minor
$g_1:=y^8-3xy^6z+3x^2y^4z^2-x^3y^2z^3\in I_8(\varphi_2)$, with monic
pure term in $y$, is obtained from the $8\times 8$ sub-matrix given by
the rows $2,3,4,5,6,7,8,9$ and the columns
$2,4,5,6,7,8,9,10$. Similarly, we get
$g_2:=t^8-3zt^6u+3z^2t^4u^2-z^3t^2u^3\in I_8(\varphi_2)$ from the
$8\times 8$ sub-matrix given by the rows $1,2,3,4,5,7,8,9$ and the
columns $3,8,10,12,13,14,15,16$. Since $(g_1,g_2,x,u)=(x,y^8,t^8,u)$,
then $\grade(g_1,g_2,x,u)=4$, $g_1,g_2$ is an $R$-regular sequence in
$I_8(\varphi_2)$, and $\grade(I_8(\varphi_2))\geq 2$. Observe that
this argument does not depend on the characteristic of the ring $R$.

As before, let us seek for minors of $\varphi_3$ with monic pure terms
in one of the parameters. For instance,
$h_1:=y^8-3xy^6z+3x^2y^4z^2-x^3y^2z^3\in I_8(\varphi_3)$ is obtained
from the $8\times 8$ sub-matrix given by the rows
$1,3,11,12,13,14,15,16$ and the columns $1,2,3,4,5,6,7,8$;
$h_2:=z^8-3yz^6t+3y^2z^4t^2-y^3z^2t^3\in I_8(\varphi_3)$ is obtained
from the rows $1,2,3,4,6,8,15,16$ and the columns $1,3,4,5,6,7,8,9$.
Finally, $h_3:=t^8-3zt^6u+3z^2t^4u^2-z^3t^2u^3 \in I_8(\varphi_3)$ is
obtained from the rows $1,2,4,5,6,7,9,11$ and the columns
$1,2,3,4,6,7,8,9$. Note that $\rad(h_1,h_2,h_3,x,u)=\mfm$. Hence
$\grade(h_1,h_2,h_3,x,u)=5$. It follows that $h_1,h_2,h_3$ is an
$R$-regular sequence in $I_8(\varphi_3)$ and that
$\grade(I_8(\varphi_3))\geq 3$.

We conclude that the complex above is a free resolution of $R/I$. It
is minimal since $\varphi_i(F_i)\subseteq \mfm\,F_{i-1}$, for every
$i=1,\ldots,4$. Therefore
\begin{eqnarray*}
4\leq \grade(I)=\min\{i\geq 0\mid \Ext^i_R(R/I,R)\neq 0\}\leq
\pd_R(R/I)\leq 4.
\end{eqnarray*}
Thus $I$ is a perfect ideal of grade $4$ and $R/I$ is Gorenstein (see,
e.g., \cite[Theorem~1.2.5, page 25]{bh} and
\cite[Proposition~3.2]{hunekeCM}).

Set $H:=(f_1,\ldots,f_7,f_9)\subset I$.  Since the aforementioned
resolution of $R/I$ is minimal, $f_8\not\in H$ and $H:f_8\subsetneq
R$. However, one can check that
$f_8^2=-xf_1f_7-xf_2f_5+xf_3f_4+xf_4^2+f_6f_9+f_7f_9$. Thus $f_8^2\in
HI$ and $HI:f_8^2=R$. Therefore, $H:f_8\subsetneq HI:f_8^2$ and $I$ is
not syzygetic (see \cite[Lemma~4.2]{planasCamb}).
\end{proof}

\begin{proposition}\label{dim5}
Let $(R,\mfm,k)$ be a regular local ring of dimension $5$. Let
$x,y,z,t,u$ be a regular system of parameters. Let
$I=(f_1,\ldots,f_9)$ be the ideal of $R$ defined as in the preceding
lemma. Then $I$ is a prime ideal of height 4, such that $R/I$ is
Gorenstein, and such that $I$ is not syzygetic.
\end{proposition}
\begin{proof}
By Lemma~\ref{perfect}, we only have to prove that $I$ is prime. Let
$(\widehat{R},\widehat{\mfm})$ be the completion of $(R,\mfm)$, which
is a five dimensional regular local ring with maximal ideal
$\widehat{\mfm}=\mfm\widehat{R}=(x,y,z,t,u)\widehat{R}$ generated by
the regular system of parameters $x,y,z,t,u$. By Lemma~\ref{perfect}
again, $I\widehat{R}=(f_1,\ldots,f_9)\widehat{R}$ is a perfect ideal
of height 4. If we prove that $I\widehat{R}$ is prime, since
$\widehat{R}$ is faithfully flat, then $I=I\widehat{R}\cap R$ and $I$
is prime as well. Therefore we can suppose that $R$ is complete.

Since $I$ is perfect of height 4, then $I$ is height unmixed and so
$\mfm$ is not an associated prime to $I$. Let $\mfp$ be any associated
prime to $I$ and set $D=R/\mfp$. Thus $D$ is a one dimensional
complete Noetherian local domain. Let $V$ be the integral closure of
$D$ in its quotient field $K$. Then $V$ is a finitely generated
$D$-module and a DVR (see \cite[Theorem~4.3.4]{sh}). Let $\nu$ be the
valuation on $K$ corresponding to $V$. Set
$(\nu_x,\nu_y,\nu_z,\nu_t,\nu_u)=
(\nu(x),\nu(y),\nu(z),\nu(t),\nu(u))$. In $V$, $f_i=0$, for
$i=1,\ldots,9$.  Applying $\nu$ to these equalities, one gets
$(\nu_x,\nu_y,\nu_z,\nu_t,\nu_u)=(6n,7n,8n,9n,10n)$, for some integer
$n\geq 1$ (see, e.g, \cite[Proof of Proposition~2.6]{gop}). In
particular, $\nu_x\geq 6$.

Let $(S,\mfn,k)$ be the regular local ring with $S=R/xR$ and
$\mfn=\mfm/xR=(y,z,t,u)$, by abuse of notation. One has
$xR+I=(x,y^2,yz,yt,yu-zt,z^2,zu,t^2,tu,u^2)$ and $R/(xR+I)\cong S/J$,
where $J$ is the ideal of $S$ defined as
$J=(y^2,yz,yt,yu-zt,z^2,zu,t^2,tu,u^2)$. Since $xR=\Ann_R(S)$, then
$\length_R(R/(xR+I))=\length_S(S/J)$. Note that $\mfn^3\subset
J\subset \mfn^2$. Consider the following two exact sequences of
$S$-modules:
\begin{eqnarray*}
0\to\mfn/J\to S/J\to S/\mfn\to 0\espai\mbox{ and }\espai
0\to \mfn^2/J\to \mfn/J\to \mfn/\mfn^2\to 0.
\end{eqnarray*}
On taking lengths,
\begin{eqnarray*}
\length_R(R/(xR+I))=\length_S(S/J)=\length_S(S/\mfn)+
\length_S(\mfn/\mfn^2)+\length_S(\mfn^2/J)=6.
\end{eqnarray*}  
Since $xR+I\subseteq xR+\mfp$ and
$R/(xR+\mfp)\cong (R/\mfp)/(x\cdot R/\mfp)=D/xD$, then
\begin{eqnarray*}
6=\length_R(R/(xR+I))\geq\length_R(R/(xR+\mfp))=\length_D(D/xD).
\end{eqnarray*}
Since $f_1=0$, $f_3+f_4=0$, $f_6+f_7=0$ and $f_9=0$ in $D$, then
$y^2,z^2,t^2,u^2\in xD$, and so $xD$ is a parameter ideal of the one
dimensional Cohen-Macaulay local domain $(D,\mfm/\mfp,k)$. Since $V$
is a finitely generated Cohen-Macaulay $D$-module of $\rank_D(V)=1$,
then $\length_D(D/xD)=\length_D(V/xV)$ (see \cite[Corollary~4.6.11,
  (c)]{bh}). Moreover
$\length_D(V/xV)=[k_V:k]\cdot\length_V(V/xV)=[k_V:k]\cdot\nu(x)$
Therefore, $\length_D(D/xD)=[k_V:k]\cdot\nu_x$. Recapitulating,
\begin{eqnarray*}
6=\length_R(R/(xR+I))\geq\length_R(R/(xR+\mfp))=[k_V:k]\cdot\nu_x\geq
6.
\end{eqnarray*}
Hence $\length_R(R/(xR+I))=\length_R(R/(xR+\mfp))$. By the additivity
of the length with respect to short exact sequences, $xR+I=xR+\mfp$.

Note that $x\not\in\mfp$, otherwise $\mfp\supset xR+I\supset
(x,y^2,z^2,t^2,u^2)$ and $\mfp=\mfm$, a contradiction. Then $\mfp\cap
xR=x\mfp$. In particular, on tensoring $0\to\mfp/I\to R/I\to R/\mfp\to
0$ by $R/xR$, one obtains the exact sequence $0\to L/xL\to R/(xR+I)\to
R/(xR+\mfp)\to 0$, where $L=\mfp/I$. Since $xR+I=xR+\mfp$, then
$L=xL$. By Nakayama's Lemma, $L=0$ and $I=\mfp$. Therefore, $I$ is a
prime ideal.
\end{proof}

We extend Proposition~\ref{dim5} to higher dimension, just by adding
to the ideal $(f_1,\ldots,f_9)$ the parameters of $R$ not involved in
the definition of the $f_i$.

\begin{corollary}\label{dimn}
Let $(R,\mfm,k)$ be a regular local ring of dimension $n\geq 6$. Let
$x_1,\ldots,x_n$ be a regular system of parameters. Let $I$ be the
ideal of $R$ generated by $f_1,\ldots,f_9,x_6,\ldots,x_n$, where
\begin{eqnarray*}
  &&f_1=x_2^2-x_1x_3\mbox{ , }f_2=x_2x_3-x_1x_4\mbox{ , }
  f_3=x_3^2-x_2x_4\mbox{,}\\ && f_4=x_2x_4-x_1x_5\mbox{ , }
  f_5=x_3x_4-x_2x_5\mbox{ , } f_6=x_4^2-x_3x_5\mbox{,}\\
  && f_7=x_3x_5-x_1^3\mbox{ , }f_8=x_4x_5-x_1^2x_2\mbox{ , }f_9=x_5^2-x_1^2x_3.
\end{eqnarray*}
Then $I$ is a prime ideal of height $n-1$, such that $R/I$ is
Gorenstein, and such that $I$ is not syzygetic.
\end{corollary}

\begin{proof} Set $I=(f_1,\ldots,f_9,x_6,\ldots,x_n)$ and
let $J=(x_6,\ldots,x_n)$ be the ideal generated by
$x_6,\ldots,x_n$. Set $\bar{R}=R/J$ and let $\bar{g}$ stand for the
class modulo $J$ of an element $g$ of $R$. Set $\bar{I}=I/J=
(\bar{f}_1,\ldots,\bar{f}_9)$. Note that $\bar{x}_1,\ldots,\bar{x}_5$
is a regular system of parameters of the regular local ring $\bar{R}$.
By Proposition~\ref{dim5}, $\bar{I}$ is a prime ideal of height 4,
such that $\bar{R}/\bar{I}$ is Gorenstein, and such that $\bar{I}$ is
not syzygetic. In particular, $I$ is a prime ideal of $R$ such that
$R/I\cong\bar{R}/\bar{I}$ is Gorenstein. Moreover,
\begin{eqnarray*}
  \height(I)=\dim(R)-\dim(R/I)=n-\dim(\bar{R}/\bar{I})=
  n-(\dim(\bar{R})-\height(\bar{I}))=n-1.
\end{eqnarray*}
Suppose that $f_8=\sum_{i=1,i\neq 8}^9a_if_i+\sum_{i=6}^nb_ix_i$, for
some $a_i,b_j\in R$. Then, on taking classes modulo $J$, $\bar{f}_8\in
(\bar{f}_1,\ldots,\bar{f}_7,\bar{f}_9)$, which is a contradiction with
the proof of Lemma~\ref{perfect}, where one shows that $\bar{I}$ is
minimally generated by $\bar{f}_1,\ldots,\bar{f}_9$. Hence $f_8\not\in
(f_1,\ldots,f_7,f_9,x_6,\ldots,x_9)=:H$ and $H:f_8\subsetneq R$. The
same equality at the end of the proof of Lemma~\ref{perfect}
shows that $f_8^2\in HI$, thus $H:f_8\subsetneq HI:f_8^2=R$. It
follows that $I$ is not syzygetic.
\end{proof}

\begin{proof}[Proof of Theorem $(ii)\Rightarrow (i)$.]
By hypothesis $(ii)$, the maximal ideal $\mfm$ of $R$ is
syzygetic. Hence $R$ is regular local (see the proof of
\cite[Corollary~3.8]{hsv}). By hypothesis $(ii)$ and using
Proposition~\ref{dim5} and Corollary~\ref{dimn}, one deduces that
$\dim(R)\leq 4$.
\end{proof}

\section{Final remarks}

From the Theorem and the isomorphism
$H_2(R,R/I,R/I)\cong\ker(\alpha_2)$ (\cite[Corollaire~15.10]{andre} or
\cite[Corollaire~3.2]{planasMan}), we deduce a characterization of
Noetherian local rings which are regular of dimension at most 4 in
terms of the Andr\'e-Quillen homology.

\begin{corollary}\label{corollaryAQ}
Let $(R,\mfm)$ be a Noetherian local ring. The following conditions
are equivalent:
\begin{itemize}
\item[$(a)$] $R$ is regular and $\dim(R)\leq 4$;
\item[$(b)$] $H_2(R,S,S)=0$ for every Gorenstein quotient domain $S$
  of $R$.
\end{itemize}
\end{corollary}

\begin{remark}
The ``global'' argument used in \cite{planasJPAA} does not seem to
work here. Indeed, let $R$ be a Noetherian ring, not necessarily
local. For the sake of easiness, let $\Spec(R)$ denote the set of
prime ideals of $R$, $\Gor(R)$ be the set of ideals $I$ of $R$ such
that $R/I$ is Gorenstein and, finally, $\Syz(R)$ be the set of
syzygetic ideals of $R$. Part of \cite[Theorem~(C)]{planasJPAA}
states: ``If $\Spec(R)\subset\Syz(R)$, then $R$ has global dimension
at most $3$''.  The proof has two steps. First, it exhibits a non
syzygetic height 3 prime ideal in any regular local ring of
dimension 4. Subsequently, it supposes that $\dim(R)\geq 4$, then it
localizes at a prime ideal $\mfq$ of height 4, obtaining the four
dimensional regular local ring $R_{\mathfrak{q}}$. Using the first
step, it deduces the existence of a prime ideal $\mfp
R_{\mathfrak{q}}\in\Spec(R_{\mathfrak{q}})$, which is not
syzygetic. Therefore $\mfp$ is prime, but not syzygetic. Thus
$\Spec(R)\not\subset\Syz(R)$, which finishes the proof. The analogous
result in dimension four would be: ``If $\Spec(R)\cap
\Gor(R)\subset\Syz(R)$, then $R$ has global dimension at most
$4$''. As a first step, we have shown Proposition~\ref{dim5} (and even
Corollary~\ref{dimn}). Trying to proceed as before, suppose that
$\dim(R)\geq 5$ and localize at a prime ideal $\mfq$ of height 5,
obtaining the five dimensional regular local ring
$R_{\mathfrak{q}}$. Using Proposition~\ref{dim5}, one deduces the
existence of a prime ideal $\mfp
R_{\mathfrak{q}}\in\Spec(R_{\mathfrak{q}})$, such that
$R_{\mathfrak{q}}/\mfp R_{\mathfrak{q}}\cong (R/\mfp)_{\mathfrak{q}}$
is Gorenstein, and such that $\mfp R_{\mathfrak{q}}$ is not
syzygetic. It follows that $\mfp$ is prime, but not syzygetic.
However, we can not assure that $R/\mfp$ is Gorenstein, in other
words, $\mfp$ is not necessarily in $\Spec(R)\cap \Gor(R)$, and so we
are not able to deduce $\Spec(R)\cap \Gor(R)\not\subset\Syz(R)$.
\end{remark}

With the preceding notations and using \cite[Theorem~(C)]{planasJPAA},
one concludes:

\begin{corollary}
Let $(R,\mfm)$ be a Noetherian local ring. Then $R$ is regular of
dimension $4$ if, and only if, $\Spec(R)\cap \Gor(R)\subset\Syz(R)$
and $\Spec(R)\not\subset\Syz(R)$.
\end{corollary}

\section*{Acknowledgement}
It is a pleasure to thank the comments of the referee. We are very
grateful to Singular, which is of inestimable help in guessing
examples and checking heavy computations.

\vspace*{0.2cm}

{\footnotesize\sc

\noindent Departament de Matem\`atiques, Universitat Polit\`ecnica de
Catalunya. \newline Diagonal 647, ETSEIB, E-08028 Barcelona, Catalunya.
Email: {\em francesc.planas@upc.edu} }
\end{document}